\documentclass[]{interact}

\usepackage{epstopdf}
\usepackage[caption=false]{subfig}

\usepackage[numbers,sort&compress]{natbib}
\bibpunct[, ]{[}{]}{,}{n}{,}{,}
\makeatletter
\def\NAT@def@citea{\def\@citea{\NAT@separator}}
\makeatother

\usepackage{xcolor}

\theoremstyle{plain}
\newtheorem{theorem}{Theorem}[section]
\newtheorem{lemma}[theorem]{Lemma}
\newtheorem{corollary}[theorem]{Corollary}

\theoremstyle{definition}
\newtheorem{definition}[theorem]{Definition}
\newtheorem{example}[theorem]{Example}

\theoremstyle{remark}

\begin{document}

\title{Solving a Class of Nonconvex Quadratic Programs by Inertial DC Algorithms}

\author{\name{Tran Hung Cuong\textsuperscript{a}, Yongdo Lim\textsuperscript{b}, Nguyen Nang Thieu\textsuperscript{c}, and Nguyen Dong Yen\textsuperscript{c}} \affil{\textsuperscript{a}Department of Computer Science, Faculty of Information Technology, Hanoi University of Industry, 298 Cau Dien, Hanoi, Vietnam; \textsuperscript{b} Department of Mathematics, Sungkyunkwan University, Suwon, South Korea; \textsuperscript{c} Institute of Mathematics, Vietnam Academy of Science and Technology, 18 Hoang Quoc Viet, Hanoi 10307, Vietnam}}

\maketitle

\begin{abstract} Two inertial DC algorithms for indefinite quadratic programs under linear constraints (IQPs) are considered in this paper. Using a qualification condition related to the normal cones of unbounded pseudo-faces of the polyhedral convex constraint set, the recession cones of the corresponding faces, and the quadratic form describing the objective function, we prove that the iteration sequences in question are bounded if the given IQP has a finite optimal value. Any cluster point of such a sequence is a KKT point. The convergence of the members of a DCA sequence produced by one of the two inertial algorithms to just one connected component of the KKT point set is also obtained. To do so, we revisit the inertial algorithm for DC programming of de Oliveira and Tcheou [de Oliveira, W.,  Tcheou, M.P.: \emph{An inertial algorithm for DC programming}, Set-Valued and Variational Analysis 2019; 27: 895--919] and give a refined version of Theorem~1 from that paper, which can be used for IQPs with unbounded constraint sets. An illustrative example is proposed.
\end{abstract} 


\section{Introduction}\label{Sect_1}

\textit{Indefinite quadratic programs under linear constraints} (IQPs) form an important class of nonconvex optimization problems. The sequential quadratic programming methods (Wilson's method, Pang's method, the local Maratos-Mayne-Polak method, global MMP method, the Maratos-Mayne-Polak-Pang method, etc.; see~\cite[Section~2.9]{Polak_1997}) transform a given twice continuously differentiable nonlinear mathematical programming problem to a sequence of IQPs. Other theoretical aspects of IQPs can be seen in~\cite{Bomze_1998}. Applications of IQPs in agriculture, economics, production operations, marketing, and public policy have been discussed by Gupta~\cite{Gupta_1995}. The role of IQPs in finance, image enhancement, training support vector machines, and machine learning is well-known (see, e.g.,~\cite{Akoa_2008,Cornuejols_2018,Liu et al_2017a,McCarl et al_1977, Xu et al_2017,Xue et al_2019})  

Qualitative properties of IQPs (solution existence, structure of the solution set, necessary and sufficient optimality conditions, and stability) are addressed in~\cite{LeeTamYen_book} and the references therein. We refer to \cite{Bomze_Danninger_1994,Cambini_Sodini_2005,PhamDinh_LeThi_2,PhamDinh_LeThi98,PhamDinh_LeThi_3,PhamDinh_LeThi_Akoa,Ye89,Ye92,Ye97} for numerical methods to solve IQPs. Quadratic programming codes and quadratic programming test examples are available in~\cite{Gould_Toint}. Most of the known algorithms stop at \textit{stationary points} (also called the \textit{Karush-Kuhn-Tucker points}, or \textit{KKT points}), or \textit{local minimizers}. This means that these algorithms are local solution methods. Taking account of the NP-hardness of general IQPs~\cite{Pardalos_Vavasis_1991}, we believe that seeking global solutions to large-scale problems remains a difficult question.

Using the DC (Difference-of-Convex functions) algorithms due to Pham Dinh and Le Thi \cite{PhamDinh_LeThi_AMV97,PhamDinh_LeThi98} (see also \cite{LeThi_PhamDinh_AOR05,PhamDinh_LeThi_4}), Pham Dinh et al.~\cite{PhamDinh_LeThi_Akoa} proposed the  {\it Projection DC decomposition algorithm} and the {\it Proximal DC decomposition algorithm} to numerically solve IQPs. Convergence, stability, and the convergence rate of these algorithms have been studied in~\cite{ATY2,CLY_2024,Tuan_JMAA2015,Tuan_JOTA2015}. Among other things, it has been proved that if the objective function is bounded below on a nonempty constraint set, then every iteration sequence converges to a KKT point and the convergence rate is $Q$-linear. To speed up the computation and force the iteration sequences to converge to KKT points of higher quality (i.e., having smaller values of the objective function), one can apply the inertial algorithm for DC programming of de Oliveira and Tcheou~\cite{de Oliveira_Tcheou_2019}. However, since the boundedness of the sublevel set related to the chosen initial point is a principal assumption of Theorems~1 and~2 in~\cite{de Oliveira_Tcheou_2019}, these fundamental results cannot be used for IQPs with unbounded constraint sets. 

This paper is threefold. First, we revisit the inertial algorithm for DC programming of de Oliveira and Tcheou and propose a refined version of Theorem~1 from~\cite{de Oliveira_Tcheou_2019}, which applies to general IQPs. Second, we specialize the algorithm to IQPs by using the two DC decomposition methods of Pham Dinh et al.~\cite{PhamDinh_LeThi_Akoa}. As a result, we obtain the Inertial Projection DC algorithm and the Inertial Proximal DC algorithm, which are denoted respectively by \textbf{InDCA1} and \textbf{InDCA2}. If the inertial parameter is null, these algorithms collapse to the above-mentioned Projection DC decomposition algorithm and Proximal DC decomposition algorithm. By imposing a novel \textit{qualification condition} on the normal cones of unbounded pseudo-faces of the polyhedral convex constraint set, the recession cones of the corresponding faces, and the quadratic form describing the objective function, we prove that all iteration sequences generated either by~\textbf{InDCA1} or by~\textbf{InDCA2} are bounded, provided that the optimal value of the IQP in question is finite. So, each iteration sequence has a cluster point, which is a KKT point. The convergence of the members of a DCA sequence produced by each of the two inertial algorithms to just one connected component of the KKT point set is also established. The third aim of our paper is to analyze the efficiency of~\textbf{InDCA1} and~\textbf{InDCA2}, as well as the assertions of the obtained convergence theorems, via one concrete example.  

Our qualification condition for IQPs resembles the well-known \textit{general position condition} in optimal control theory of linear control systems~\cite[p.~116]{PBGM_1962}. 

The organization of the paper is as follows. Some preliminaries and auxiliary results are presented in Section~\ref{Preliminaries}. Then, in Section~\ref{Sect_3}, we consider two inertial DC algorithms for IQPs. Finally, in Section~\ref{Sect_4}, we obtain convergence theorems for the new algorithms and study a useful illustrative example.

\section{Preliminaries and Auxiliary Results}\label{Preliminaries}

The set of natural numbers is denoted by $\mathbb N$. The scalar product (resp., the norm) of two vectors $x,y\in\mathbb R^n$ (resp., of a vector  $x\in\mathbb R^n$) is denoted by $\langle x,y\rangle$ (resp., $\|x\|$). Let $B(x,\varepsilon)$ (resp., $\bar{B}(x,\varepsilon)$) stand for the open (resp., closed) ball with center $x$ and radius $\varepsilon>0$. The convention $(+\infty)-(+\infty)=+\infty$ will be used throughout the forthcoming. The distance from $x\in\mathbb R^n$ to a subset $\Omega$ of $\mathbb R^n$ is defined by $$d(x,\Omega)=\inf\{\|x-u\|\mid u\in\Omega\}.$$

We now recall some standard notions and notations from convex analysis and optimization theory, which can be found in~\cite{Bonnans_Shapiro_2000,Clarke_1990,Ioffe_Tihomirov_1979,Roc70} (see also~\cite{ CLY_2024, de Oliveira_Tcheou_2019} and the references therein). 
 
The \textit{normal cone} to a convex set $\Omega\subset\mathbb R^m$ at $x\in \Omega$ is the set 
$$N_\Omega(x):=\big\{\xi\in\mathbb R^n \mid \langle\xi,v-x\rangle\leq 0\ \; \forall v\in \Omega\big\}.$$

A nonzero vector $v\in\mathbb R^n$ is said to be a \textit{direction of recession}  of a nonempty convex set $\Omega\subset\mathbb R^n$ if $x+tv\in \Omega$ for every $t\geq 0$ and every $x\in \Omega.$ The set composed by $0\in\mathbb R^n$ and all the directions $v\in\mathbb R^n\setminus\{0\}$ satisfying the last condition, is called the \textit{recession cone} of $\Omega$ and denoted by $0^+\Omega.$  If $\Omega$ is closed and convex, then  $$0^+\Omega=\{v\in\mathbb R^n\mid \exists x\in \Omega\ \, {\rm such\ that}\ \, x+tv\in \Omega\ \, {\rm for\ all}\ \, t> 0\}.$$

For a function $f:\mathbb R^n\to \mathbb R\cup \{+\infty\}$, one defines the \textit{domain} of $f$ and the \textit{epigraph} of $f$ respectively by ${\rm dom}\,f=\big\{x\in\mathbb R\mid f(x)<+\infty\big\}$ and $${\rm epi}\,f=\{(x,\mu)\in \mathbb R^n\times \mathbb R\mid f(x)\leq\mu\}.$$
If $f$ is convex (i.e., ${\rm epi}\,f$ is convex), then the \textit{subdifferential} of $f$ at $\bar x\in {\rm dom}\,f$ is defined by 
$$\partial f(\bar x)=\big\{x^*\in\mathbb R^n\mid f(x)\geq f(\bar x)+\langle x^*,x-\bar x\rangle\ \;\forall x\in\mathbb R^n\big\}.$$ Given any $\varepsilon\geq 0$, one defines the \textit{$\varepsilon-$subdifferential} of $f$ at $\bar x\in {\rm dom}\,f$ by   
$$\partial_\varepsilon f(\bar x)=\big\{x^*\in\mathbb R^n\mid f(x)\geq f(\bar x)+\langle x^*,x-\bar x\rangle-\varepsilon\ \;\forall x\in\mathbb R^n\big\}.$$ For any $\bar x\notin {\rm dom}\,f$, one puts $\partial f(\bar x)=\emptyset$ and $\partial_\varepsilon f(\bar x)=\emptyset$. 

A convex function $f:\Omega\to\mathbb R$, where $\Omega\subset\mathbb R^n$ is a convex set, is said to be \textit{strongly convex} on  $\Omega$ with modulus $\rho>0$ if for any $x,y\in \Omega$ and $ t\in(0,1)$ one has	
\begin{equation}
	\label{sc}
	f\big[(1- t)x + t y\big] \leq (1- t) f(x) + t f(y) -\dfrac{\rho}{2} t(1- t)\|x-y\|^2.
\end{equation} 

Since the next characterization of strong convexity sometimes is employed as the definition of the latter concept (see, e.g.,~\cite[Assumption~A1]{de Oliveira_Tcheou_2019}), a detailed proof is given here for the clarity of our presentation.
 
\begin{lemma}
\label{prop_1}
Let $f:\Omega\to\mathbb R$, where $\Omega\subset\mathbb R^n$ is a nonempty open convex set, be a convex function. Then, $f$ is strongly convex on $\Omega$ with modulus $\rho>0$ if and only if for any $x,y\in \Omega$ and for any $x^*\in\partial f(x)$ one has
\begin{equation}
	\label{sc_1}
	f(y) \geq f(x)+\langle x^*,y-x\rangle +\dfrac{\rho}{2}\| y-x\|^2.
\end{equation}
\end{lemma}

\begin{proof}
Suppose that $f$ is strongly convex on $\Omega$ with modulus $\rho>0$. Take any $x,y\in\Omega$, $x^*\in\partial f(x)$, and $ t\in(0,1)$. It follows from~\eqref{sc} that
\begin{equation*}
	t f(y) - t f(x) \geq f\big[(1- t) x + t y\big] -f(x)+\dfrac{\rho}{2} t(1- t)\|x-y\|^2.
\end{equation*}
Since $f\big[(1- t) x + t y\big] -f(x)\geq \langle x^*,  t y -
t x\rangle$, we get
\begin{equation*}
	t f(y) - t f(x) \geq \langle x^*,  t y -
	t x\rangle +\dfrac{\rho}{2} t(1- t)\|x-y\|^2.
\end{equation*}
Dividing both sides of this inequality by $ t$ yields
\begin{equation*}
	f(y) - f(x) \geq  \langle x^*, y -
	x\rangle +\dfrac{\rho}{2}(1- t)\|x-y\|^2.
\end{equation*}
Hence, passing $t\to 0$, from the last inequality we obtain~\eqref{sc_1}.

Now suppose that~\eqref{sc_1} is satisfied for all $x,y\in \Omega$ and $x^*\in\partial f(x)$. By our assumptions, $\partial f(x)$ is nonempty for every $x\in\Omega$ (see~\cite[Theorem~23.4]{Roc70}). Fix any $x,y\in \Omega$, $x^*\in\partial f(x)$, and $ t\in (0,1)$. Put $z=(1- t)x+ t y$ and take $z^*\in\partial f(z)$. By~\eqref{sc_1} we have
$$f(x)-f(z) \geq \langle z^*,x-z\rangle +\dfrac{\rho}{2}\| x-z\|^2,$$
and
$$f(y)-f(z) \geq  \langle z^*,y-z\rangle +\dfrac{\rho}{2}\| y-z\|^2.$$
By multiplying both sides of the first inequality (resp., both sides of the second inequality) by $(1- t)$ (resp., by $t$), then adding the results, we obtain
$$\begin{array}{rl}
	& (1- t)\big[f(x)-f(z)\big] + t\big[f(y)-f(z)\big]\\ 
	& \geq \big\langle z^*,(1- t)(x-z)+ t (y-z)\big\rangle +\dfrac{\rho}{2}\left[(1- t)\| x-z\|^2 +t\|y-z\|^2\right].
\end{array}$$
Since $x-z=  t(x-y)$ and $y-z=(1- t)(y-x)$, this yields
$$(1- t)\big[f(x)-f(z)\big] + t\big[f(y)-f(z)\big] \geq \langle z^*,0\rangle +\dfrac{\rho}{2}\big[ (1-t)t^2+t(1- t)^2\big] \|x-y \|^2.$$
As $(1-t)t^2+t(1- t)^2=t(1-t)$, the last inequality implies~\eqref{sc}. 

The proof is complete.
\end{proof}

Given $s$ vectors $v^1,\dots,v^s$ in $\mathbb R^n$, we denote by ${\rm pos}\{v^1,\dots,v^s\}$ the closed convex cone generated by $v^1,\dots,v^s$, that is $${\rm pos}\{v^1,\dots,v^s\}= \Big\{v=\sum_{i=1}^s\lambda_iv^i\mid \lambda_i\geq 0\ \,{\rm for}\ \,i=1,\dots,s\Big\}.$$ 

The \textit{metric projection} of $u\in\mathbb R^n$ onto a nonempty closed convex subset $\Omega$ is denoted by $P_\Omega(u)$. Thus, $P_\Omega(u)\in \Omega$ and $\big\|u-P_\Omega(u)\big\|=\displaystyle\min_{x\in \Omega} \|u-x\|.$

By $I$ we denote the unit matrix in $\mathbb R^{n\times n}$. For a symmetric matrix $Q\in\mathbb R^{n\times n}$, its  \textit{eigenvalues} are ordered in the sequence $\lambda_1(Q)\leq ...\leq\lambda_n(Q)$ with counting multiplicities. 

\subsection{Inertial DC Algorithms}

Following de Olivera and Tcheou~\cite{de Oliveira_Tcheou_2019}, we consider nonconvex optimization problems of the form
\begin{equation}\label{DC_program_1}
\min\big\{f(x):=f_1(x)-f_2(x)\mid x\in\mathbb R^n\big\}
\end{equation} with $f_1, f_2 : \mathbb R^n \to \mathbb R\cup\{+\infty\}$ being such convex and possibly nonsmooth functions that there is an open and convex set $\Omega$ in $\mathbb R^n$ satisfying
\begin{equation}\label{domains}
	{\rm dom}\, f_1\subset\Omega\subset	{\rm dom}\, f_2.
\end{equation} Besides, it is assumed that the function $f_1$ is \textit{lower semicontinuous} on $\mathbb R^n$, i.e., for any $\bar x\in \mathbb R^n$ and for any $\varepsilon>0$, there exists a neighborhood $U$ of $\bar x$ such that $f_1(x)\geq f_1(\bar x)-\varepsilon$ for all $x\in U$. The latter condition is equivalent to the requirement saying that \textit{both sets ${\rm dom}\, f_1$ and ${\rm epi}\,f_1$ are closed}. 

If $f(\bar x)\in\mathbb R$ and there is a neighborhood $U$ of $\bar x$ such that $f(\bar x)\leq f(x)$ for all $x\in U$, then $\bar x$ is called a \textit{local solution} of~\eqref{DC_program_1}.  If $f(\bar x)\in\mathbb R$ and $f(\bar x)\leq f(x)$ for all $x\in \mathbb R^n$, then $\bar x$ is said to be a \textit{global solution} of~\eqref{DC_program_1}. One says that $\bar x\in\Omega$ is a \textit{stationary point} (resp., a \textit{critical point}) of~\eqref{DC_program_1} if $\partial f_2(\bar x)\subset \partial f_1(\bar x)$ (resp., $\partial f_1(\bar x)\cap \partial f_2(\bar x)\neq\emptyset$). By $\mathcal{S},\mathcal{S}_1,\mathcal{S}_2$, and $\mathcal{S}_3$, respectively, we denote the sets of the global solutions, the local solutions, the stationary points, and the critical points of~\eqref{DC_program_1}. The fundamental necessary optimality condition in DC programming (see~\cite{de Oliveira_Tcheou_2019, de Oliveira_2020} and the references therein) asserts that $\mathcal{S}_1\subset \mathcal{S}_2$. In addition, since $\partial f_2(x)\neq\emptyset$ for every $x\in\Omega$  by~\cite[Theorem~23.4]{Roc70}, it holds that
 \begin{equation}\label{basic_inclusions}
 	\mathcal{S}\subset\mathcal{S}_1\subset\mathcal{S}_2\subset \mathcal{S}_3\subset {\rm dom}\, f_1. 
 \end{equation} By constructing suitable examples, one can show that each inclusion in~\eqref{basic_inclusions} can be strict. If the function $f_2$ is Fr\'echet differentiable on $\Omega$, then $\partial f_2(x)=\{\nabla f_2(x)\}$ for all $x\in\Omega$. In that case, one has $\mathcal{S}_2=\mathcal{S}_3$.

The inertial DC algorithmic pattern suggested in~\cite{de Oliveira_Tcheou_2019} is as follows.

\medskip
\textbf{Algorithm~1 (The general algorithm)}

\noindent \rule[0.05cm]{14.6cm}{0.01cm}

1:\ \, Let $x^0\in {\rm dom}\, f_1$, ${\rm Tol}\geq 0, \lambda\in [0,1)$ and $\gamma\in \big[0,2^{-1}(1-\lambda)\rho\big)$ be given. Set $x^{-1}=x^0$

	2:\ \, \textbf{for} $k=0,1,2,\dots $ \textbf{do}
	
	3:\ \,\hskip0.5cm Set $d^k=\gamma(x^k-x^{k-1})$ and find $x^{k+1}\in\Omega$ such that
	\begin{equation}\label{eq(8)}
		\partial_{\varepsilon^{k+1}}f_1(x^{k+1})\cap \left[\partial f_2(x^k)+d^k\right]\neq\emptyset\ \; {\rm with}\ \; 0\leq\varepsilon^{k+1}\leq 2^{-1}\lambda\rho\|x^{k+1}-x^k\|^2
	\end{equation}
	
		4:\ \, \hskip0.5cm \textbf{if} $\|x^{k+1}-x^k\|\leq {\rm Tol}$ and $\|d^k\|\leq {\rm Tol}$ \textbf{then}
		
		5:\ \, \hskip1cm Stop and return $(x^k,f(x^k))$
		
		6:\ \, \hskip0.5cm  \textbf{end if}
		
		7:\ \, \textbf{end for}\\
\noindent \rule[0.05cm]{14.6cm}{0.01cm}

\medskip
If $f_1$ is continuously Fr\'echet differentiable on ${\rm dom}\,f_1$, then $\partial f_1(x)=\{\nabla f_1(x)\}$ for all $x\in {\rm dom}\,f_1$. In that case, the following alternative to condition~\eqref{eq(8)} has been proposed in~\cite{de Oliveira_Tcheou_2019}: Compute $x^{k+1}\in {\rm dom}\,f_1$ such that
\begin{equation}\label{eq(9)}
	\|\nabla f_1(x^{k+1})- \big(\xi^k_2+d^k\big)\|\leq 2^{-1}\lambda\rho\|x^{k+1}-x^k\| \ \; {\rm for\ some}\ \; \xi^k_2\in \partial f_2(x^k).
\end{equation}

Since sublevel sets of indefinite QP problems under linear constraints can be unbounded, the convergence theorem for Algorithm~1 in~\cite{de Oliveira_Tcheou_2019} is not completely suited for our research purposes. So, we have to slightly modify the formulation and proof of the just-mentioned fundamental result. 
 \textit{In what follows, if not otherwise stated, we take ${\rm Tol}=0$.}

\begin{theorem}\label{Thm_convergence_1} {\rm (cf.~\cite[Theorem~1]{de Oliveira_Tcheou_2019})} Consider Algorithm~1 and assume that $f_2$ is strongly convex on~$\Omega$ with a parameter $\rho>0$, $x^0\in {\rm dom}\, f_1$, $$\lambda\in [0,1),\ \;\gamma\in \big[0,2^{-1}(1-\lambda)\rho\big).$$ Put $\alpha=2^{-1}[(1-\lambda)\rho-\gamma]$, $\alpha_1=\alpha-2^{-1}\gamma$, and suppose that $f$ is bounded below on $\Omega$ by a constant~$\beta\in\mathbb R$. Then, the following assertions are valid.
\begin{itemize}
	\item[{\rm (a)}] For all $k\in\mathbb N$, 
	\begin{equation}\label{est1_ini}
		f(x^{k+1})+\alpha \|x^{k+1}-x^k\|^2\leq \left[f(x^k)+\alpha \|x^k-x^{k-1}\|^2\right]-\alpha_1\|x^k-x^{k-1}\|^2.
	\end{equation}
	\item[{\rm (b)}] It holds that
	\begin{equation}\label{est2_ini}
		0\leq\alpha_1\sum_{k=0}^\infty \|x^{k+1}-x^k\|^2\leq f(x^0)-\beta.
	\end{equation}
	\item[{\rm (c)}] One has $\lim\limits_{k\to\infty}\|x^{k+1}-x^k\|=0$.
	\item[{\rm (d)}] Any cluster point of the sequence $\{x^k\}$ belongs to the set $\mathcal{S}_3$ of the critical points of~\eqref{DC_program_1}.
\end{itemize}
\end{theorem}
\begin{proof} Clearly, the condition $\gamma\in \big[0,2^{-1}(1-\lambda)\rho\big)$ implies that $\alpha\geq\alpha_1>0$. Assertion~(a) follows from~\cite[Lemma~2]{de Oliveira_Tcheou_2019}. To prove assertion~(b), take an arbitrary positive integer $m$ and add the first $m$ inequalities given by~\eqref{est1_ini} side by side to get 
	$$f(x^m)+\alpha \|x^m-x^{m-1}\|^2\leq f(x^0)-\alpha_1\|x^1-x^0\|^2-\alpha_1\|x^2-x^1\|^2-\ldots-\alpha_1\|x^{m-1}-x^{m-2}\|^2.$$ This yields
	$$\alpha_1\sum\limits_{k=0}^{m}\|x^{k+1}-x^{k}\|^2\leq f(x^0)-f(x^m).$$
	Hence,
	\begin{equation}\label{est3_ini}
		0\leq \alpha_1\sum\limits_{k=1}^{m-1}\|x^k-x^{k-1}\|^2\leq f(x^0)-\beta.
	\end{equation} Passing~\eqref{est3_ini} to the limit as $m\to\infty$ gives~\eqref{est2_ini}.
	
	 Assertion~(c) is an immediate consequence of the convergence of the series $\sum\limits_{k=0}^\infty \|x^{k+1}-x^k\|^2$ in~\eqref{est2_ini}. 
	
	Finally, to justify assertion~(d), suppose that $\{x^{k'}\}$ is a subsequence of $\{x^k\}$ with $\lim\limits_{k'\to\infty} x^{k'}=\bar x$. From~\eqref{eq(8)} and~\eqref{eq(9)} we have $\{x^k\}\subset {\rm dom}\, f_1$. Hence, the closedness of ${\rm dom}\, f_1$ assures that $\bar x\in {\rm dom}\, f_1$. Therefore, by~\eqref{domains} we have $\bar x\in\Omega$. Let $\delta>0$ be such that $\bar B(\bar x,\delta)\subset\Omega$. By considering a subsequence of $\{x^{k'}\}$ (if necessary), we may assume that $\{x^{k'}\}\subset \bar B(\bar x,\delta)$. Since $\Omega$ is open and $\Omega\subset	{\rm dom}\, f_2$, by~\cite[Theorem~23.4]{Roc70} we know that $\partial f_2(x)$ is nonempty and bounded for every $x\in\Omega$. Combining this with the closedness of the subdifferential, we see that $\partial f_2(x)$ is nonempty and compact for every $x\in\Omega$.  Furthermore, according to~Corollary at page~35 and Proposition~2.2.7 in~\cite{Clarke_1990}, $\partial f_2(x)$ coincides with the Clarke subdifferential of the locally Lipschitz function $f_2$. So, the set-valued map $\partial  f_2:\Omega\rightrightarrows \mathbb R^n$ is upper semicontinuous by~\cite[Proposition~2.6.2]{Clarke_1990}. Therefore, by the preservation of compactness of upper semicontinuous set-valued maps having compact values, we can assert that the set $G:=\bigcup\limits_{x\in \bar B(\bar x,\delta)}\partial f_2(x)$ is compact. By~\eqref{eq(8)}, for each $k\in\mathbb R^n$ we can select $\xi_1^{k+1}\in 	\partial_{\varepsilon^{k+1}}f_1(x^{k+1})$ and $\xi_2^k\in  \partial f_2(x^k)$ such that 
	\begin{equation}\label{eq(8a)}
	\xi_1^{k+1}=\xi_2^k+d^k.
	\end{equation} If the rule~\eqref{eq(9)} applies, then for every index $k'$ we have 
\begin{equation}\label{eq(9a)}
	\|\nabla f_1(x^{k'+1}) - \big(\xi^{k'}_2+d^{k'}\big)\|\leq 2^{-1}\lambda\rho\|x^{k'+1}-x^{k'}\|,
\end{equation} where $\xi_2^{k'}\in\partial f_2(x^{k'})$. Since $\{\xi^{k'}_2\}\subset G$, by considering a subsequence of $\{x^{k'}\}$ (if necessary), we may assume that $\lim\limits_{k'\to\infty}\xi^{k'}_2=\bar\xi_2\in Q$. Then, by upper semicontinuity of the set-valued map $\partial  f_2:\Omega\rightrightarrows \mathbb R^n$ we get $\bar\xi_2\in \partial  f_2(\bar x)$. 

First, suppose that the rule~\eqref{eq(8)} applies. Then, the inclusion $\xi_1^{k'+1}\in \partial_{\varepsilon^{k'+1}}f_1(x^{k'+1})$ implies that
$$f_1(x)\geq f_1(x^{k'+1})+\langle\xi_1^{k'+1},x-x^{k'+1}\rangle-\varepsilon^{k'+1}\quad \forall x\in\mathbb R^n,\ \forall k'.$$ Thus, for each $x\in\mathbb R^n$ and for any $\varepsilon>0$, using the lower semicontinuity of $f_2$ at $\bar x$ and~\eqref{eq(8a)}, we have 
$$f_1(x)\geq \big[f_1(\bar x)-\varepsilon\big] +\langle \xi_2^{k'}+d^{k'},x-x^{k'+1}\rangle-\varepsilon^{k'+1}$$ for all $k'$ large enough. Hence, thanks to assertion~(c), passing the last inequality to limit as $k'\to\infty$ and remembering that $$0\leq\varepsilon^{k'+1}\leq 2^{-1}\lambda\rho\|x^{k'+1}-x^{k'}\|^2$$
yield $f_1(x)\geq \big[ f_1(\bar x)-\varepsilon\big] +\langle \bar\xi_2,x-\bar x\rangle.$ Since $\varepsilon$ was taken arbitrarily, it follows that 
$$f_1(x)\geq f_1(\bar x)+\langle \bar\xi_2,x-\bar x\rangle\quad \forall x\in\mathbb R^n.$$ Consequently, we have $\bar\xi_2\in\partial  f_1(\bar x)\cap \partial  f_2(\bar x)$. Thus, $\bar x\in \mathcal{S}_3$.

Now, suppose that the rule~\eqref{eq(9)} is effective. Since $f_1$ is continuously Fr\'echet differentiable on ${\rm dom}\,f_1$, letting $k'\to\infty$ and using assertion~(c), from~\eqref{eq(9a)} we obtain $\nabla f_1(\bar x)=\bar\xi_2$. This means that $\bar\xi_2\in\partial  f_1(\bar x)\cap \partial  f_2(\bar x)$; so, $\bar x\in \mathcal{S}_3$.

The proof is complete.
\end{proof}

\begin{corollary}\label{Corr1} In addition to the assumptions of Theorem~\ref{Thm_convergence_1}, suppose that the sublevel set $L_f(x^0):=\big\{x\in\Omega\mid f(x)\leq f(x^0)\big\}$ of $f$ is bounded. Then the iterative sequence $\{x^k\}$ generated by Algorithm~1 has at least one cluster point which is a critical point of~\eqref{DC_program_1}.
\end{corollary}
\begin{proof} Assertion~(a) of Theorem~\ref{Thm_convergence_1} implies that
\begin{equation*}\label{est1_ini}
		f(x^{k+1})+\alpha \|x^{k+1}-x^k\|^2\leq \left[f(x^0)+\alpha \|x^0-x^{-1}\|^2\right]-\alpha_1\|x^0-x^{-1}\|^2=f(x^0)
	\end{equation*} every $k\in\mathbb N$. So, the whole sequence $\{x^k\}$ is contained in $L_f(x^0)$. As the later set is bounded, $\{x^k\}$ must have at least one convergent subsequence. By assertion~(d) of Theorem~\ref{Thm_convergence_1}, the limit point of that subsequence is a critical point of~\eqref{DC_program_1}.	
\end{proof}

\begin{corollary}\label{Corr2} If the assumptions of Theorem~\ref{Thm_convergence_1} are satisfied, then Algorithm~1 stops after a finite number of iterations if one takes ${\rm Tol}>0$. \end{corollary}
\begin{proof} By assertion~(c) of Theorem~\ref{Thm_convergence_1} and by the equality $d^k=\gamma(x^k-x^{k-1})$, there exists a natural number $\bar k$ such that $\max\big\{\|x^{k+1}-x^k\|, \|d^k\|\big\}\leq {\rm Tol}$  for every $k\geq \bar k$. In particular, $\|x^{\bar k+1}-x^{\bar k}\|\leq {\rm Tol}$ and $\|\bar d^k\|\leq {\rm Tol}$. Therefore, Algorithm~1 stops at the step number $\bar k$, or earlier.   
\end{proof}

\subsection{Indefinite QP problems}

Consider \textit{the indefinite quadratic programming problem under linear constraints} (called the IQP for brevity):
\begin{eqnarray}\label{QP problem}
	\min\Big\{\dfrac{1}{2}x^TQx+q^Tx \mid Ax\geq b\Big\}, \end{eqnarray}
where $Q\in\mathbb R^{n\times n}$ and $A\in\mathbb  R^{m\times n}$ are given matrices, $Q$  is symmetric, $q\in\mathbb  R^n$ and $b\in\mathbb  R^m$ are given vectors, and the superscript $^T$ denotes matrix transposition. The constraint set of~\eqref{QP problem} is 
\begin{equation}\label{C}
C:=\big\{x\in\mathbb R^n\mid Ax\geq b\big\}.
\end{equation} Clearly, $C$ is convex and closed. Assume that $C\neq\emptyset$. Since $x^TQx$ is an indefinite quadratic form, the objective function $f(x)$ is nonconvex in general; hence (\ref{QP problem}) is a nonconvex optimization problem. 

For an index set $\alpha\subset \{1,\dots,m\}$, by $A_\alpha$ we denote the matrix composed by the rows $A_i$, $i\in\alpha$, of $A$. Similarly, $b_\alpha$ is the vector composed by the components $b_i$, $i\in\alpha$, of $b$. The {\it pseudo-face} of $C$ corresponding to $\alpha$ is the set $${\mathcal F}_\alpha:=\big\{x\in\mathbb R^n\mid A_\alpha x=b_\alpha,\ A_{\bar\alpha}x>b_{\bar\alpha}\big\},$$ where $\bar\alpha:=\{1,\dots,m\}\setminus\alpha$. The {\it face} of $C$ corresponding to ${\mathcal F}_\alpha$ is the set  $$F_\alpha:=\big\{x\in\mathbb R^n\mid A_\alpha x=b_\alpha,\ A_{\bar\alpha}x\geq b_{\bar\alpha}\big\}.$$ It is well-known and easily verified that, for any $x\in {\mathcal F}_\alpha$, one has $$N_C(x)={\rm pos}\{A_i^T\mid i\in\alpha\}.$$ This means that the normal cone operator $N_C(.):\mathbb R^n\rightrightarrows\mathbb R^n$ is constant on every pseudo-face of $C$. By $N_{{\mathcal F}_\alpha}$ we denote the normal cone to $C$ at any point $x\in {\mathcal F}_\alpha$ and we say, by abuse of terminology, that $N_{{\mathcal F}_\alpha}$ is the \textit{normal cone of the pseudo-face}~${\mathcal F}_\alpha$.

To solve~\eqref{QP problem} using Algorithm~1, we define
\begin{eqnarray}\label{components_IQP} f_1(x)=\Big[\frac{1}{2}x^TQ_1x+q^Tx\Big]+\delta_C(x)\ \; {\rm and}\ \; f_2(x)=\frac{1}{2}x^TQ_2x,
\end{eqnarray} where $Q=Q_{1}-Q_{2}$ with $Q_1$ and $Q_2$ being symmetric positive definite matrices, and $\delta_C(x)$ is the \textit{indicator function} of $C$ (i.e., $\delta_C(x)=0$ for $x\in C$ and $\delta_C(x)=+\infty$ for $x\notin C$). Later, we will see how to choose the matrices $Q_1$ and $Q_2$ in some appropriate manners. Clearly, $f(x):=\dfrac{1}{2}x^TQx+q^Tx+\delta_C(x)$ has the DC decomposition $f(x)=f_1(x)-f_2(x)$. Therefore, with these functions $f$, $f_1$, and $f_2$, the IQP in~\eqref{QP problem} is equivalent to the nonsmooth DC program~\eqref{DC_program_1}. Since ${\rm dom}\, f_1=C$ and ${\rm dom}\, f_2=\mathbb R^n$, condition~\eqref{domains} is satisfied with $\Omega:=\mathbb R^n$. Moreover, the function $f_1$ is lower semicontinuous on $\mathbb R^n$.
 
\begin{definition} {\rm For $x\in\mathbb R^n$, if there exists a multiplier $\lambda\in\mathbb R^m$  such that
		\begin{eqnarray*}\label{KKT_Point_Set}
			\begin{cases}Qx+q-A^T\lambda=0,\\ Ax\geq b,\ \; \lambda\geq 0,\ \; \lambda^T(Ax-b)=0,\end{cases}
		\end{eqnarray*} then $x$ is said to be a {\it Karush-Kuhn-Tucker point} (a {\it KKT point}) of the IQP.
} \end{definition}

Denote the \textit{KKT point set} of~\eqref{QP problem} by $C_*$ and observe (see, e.g.,~\cite{LeeTamYen_book}) that $x\in C_*$ if and only if $x$ is a solution of the \textit{affine variational inequality}
\begin{eqnarray}\label{AVI} x\in C,\quad \langle Qx+q,u-x\rangle\geq 0\ \; \forall u\in C.
\end{eqnarray} Note that condition~\eqref{AVI} means that $0\in  Qx+q+N_C(x)$, and this inclusion can be rewritten equivalently as $\nabla f_2(x)\in\partial f_1(x)$. Hence, the KKT point set $C_*$ of the IQP coincides with the stationary point set ${\mathcal S}_2$ of~\eqref{DC_program_1}. Moreover, since the function $f_2$ is Fr\'echet differentiable on~$\Omega$, specializing the inclusions in \eqref{basic_inclusions} to the IQP, we have   \begin{equation}\label{basic_inclusions_IQP}
\mathcal{S}\subset\mathcal{S}_1\subset C_*=\mathcal{S}_2= \mathcal{S}_3\subset C. 
\end{equation}

The following lemma describes the behavior of $f$ on $C_*$.

\begin{lemma}\label{Lemma 3} {\rm (\cite[Lemma 3.1]{Luo_Tseng_1992}; see also \cite[Lemma 2.2]{Tuan_JMAA2015})}\ \,
	Let $C_1, C_2, \cdots, C_r$ denote the connected components of $C_*$. Then we have
	$ C_*=\displaystyle\bigcup_{i=1}^r C_i,$
	and the following properties are valid.
	\begin{itemize}
	\item[{\rm (a)}] Each $C_i$ is the union of finitely many polyhedral convex sets.
	\item[{\rm (b)}] The closed sets $C_i$, $i=1,\ldots r$, are properly separated each from others, that is, there exists $\delta>0$ such that if  $i\neq j$ then
	\begin{equation}\label{positive_excess}
		d(x,C_j)\geq \delta\quad \forall x\in C_i.
	\end{equation}
	\item[{\rm (c)}] The objective function of~\eqref{QP problem} is constant on each $C_i$.
\end{itemize}
\end{lemma}



Since $Q$ is symmetric, its eigenvalues are real. As noted by Cuong et al.~\cite{CLY_2024}, \textit{the smallest eigenvalue $\lambda_1(Q)$ and the largest eigenvalue $\lambda_n(Q)$ of $Q$ can be computed easily by some algorithm} (for instance, by the Newton-Raphson algorithm in~\cite{Stoer_Bulirsch_1980}) \textit{or software}. So, the next choices of $Q_1$ and $Q_2$, which were proposed by Pham Dinh et al.~\cite{PhamDinh_LeThi_Akoa}, can be easily done: 

(a) $Q_1:=\eta I$, $Q_2:=\eta I-Q$, where $\eta$ is a positive real value satisfying the condition $\eta>\lambda_n(Q)$;

(b) $Q_1:=Q+\eta I$, $Q_2:=\eta I$, where $\eta$ is a positive real value satisfying the condition $\eta>-\lambda_1(Q)$.

Both choices guarantee that $Q=Q_{1}-Q_{2}$, where the matrices $Q_1$ and $Q_2$ are symmetric and positive definite. Following~\cite{CLY_2024}, we call the number $\eta$ \textit{the decomposition parameter}. Using the subdifferential sum rule for convex functions in~\cite[Theorem~23.8]{Roc70}, from~\eqref{components_IQP} we get 
\begin{equation}\label{subdiff_IQP}
\partial f_1(x)=(Q_1x+q)+N_C(x)\ \; {\rm and}\ \; \partial f_2(x)=\{Q_2x\}
\end{equation} for all $x\in\mathbb R^n$.

If the choice~(a) is used, then by~\eqref{subdiff_IQP} we have
\begin{equation}\label{subdiff_IQP_(a)}
	\partial f_1(x)=(\eta x+q)+N_C(x)\ \; {\rm and}\ \; \partial f_2(x)=\{\eta x-Qx\}.
\end{equation} Meanwhile, if the choice~(b) is adopted, then~\eqref{subdiff_IQP} yields
\begin{equation}\label{subdiff_IQP_(b)}
\partial f_1(x)=[(Q+\eta I)x+q]+N_C(x)\ \; {\rm and}\ \; \partial f_2(x)=\{\eta x\}.
\end{equation}

 If we choose $\lambda=0$ for Algorithm~1, then $\varepsilon^k=0$ for all $k\in\mathbb N$ and condition~\eqref{eq(8)} becomes \begin{equation}\label{eq(8)_IQP}
\nabla f_2(x^k)+d^k\in \partial f_1(x^{k+1}).
\end{equation} Combining~\eqref{eq(8)_IQP} with~\eqref{subdiff_IQP_(a)} and the formula $d^k=\gamma \big(x^k-x^{k-1}\big)$ gives
$$\eta x^k-Qx^k+\gamma(x^k-x^{k-1})\in \big(\eta x^{k+1} +q\big)+N_C(x^{k+1}).$$ It follows that
\begin{equation}\label{basic(a)} \gamma \big(x^k-x^{k-1}\big)\in \eta \big(x^{k+1}-x^k\big)+\big(Qx^k+q\big)+N_C(x^{k+1}).
\end{equation}  Similarly, since $d^k=\gamma \big(x^k-x^{k-1}\big)$, a combination of~\eqref{subdiff_IQP_(b)} and~\eqref{eq(8)_IQP}  yields
$$\eta x^k+\gamma(x^k-x^{k-1})\in \Big[(Q+\eta I)x^{k+1} +q\Big]+N_C(x^{k+1}).$$ Therefore,
\begin{equation}\label{basic(b)} \gamma \big(x^k-x^{k-1}\big)\in \eta \big(x^{k+1}-x^k\big)+\big(Qx^{k+1}+q\big)+N_C(x^{k+1}).
\end{equation}  

\textit{From now on, we choose $\lambda=0$ for Algorithm~1.} 

\section{Two Inertial DC Algorithms for IQPs}\label{Sect_3}

The analysis in Subsection~2.2 gives us two inertial DC algorithms to solve the IQP in~\eqref{QP problem}, which are termed \textit{Inertial Projection DCA} and \textit{Inertial Proximal DCA}. These algorithms will be abbreviated respectively to \textbf{InDCA1} and \textbf{InDCA2}.

\subsection{Inertial Projection DCA}

Condition~\eqref{basic(a)} can be rewritten equivalently as
\begin{equation*}\label{AVI(a)} \left\langle \left[\Big(1+\frac{\gamma}{\eta}\Big)x^k-\frac{\gamma}{\eta}x^{k-1}-\frac{1}{\eta}\big(Qx^k+q\big)\right]-x^{k+1},x-x^{k+1} \right\rangle\leq 0\quad\forall x\in C.
\end{equation*} By the well-known characterization of the metric projection onto a closed convex set in~\cite[Chapter~I, Theorem~2.3]{KS1980} and the inclusion $x^{k+1}\in C$, this means that
 \begin{equation}\label{iteration_InDCA1} 
 	x^{k+1}=P_C\left(\Big(1+\frac{\gamma}{\eta}\Big)x^k-\frac{\gamma}{\eta}x^{k-1}-\frac{1}{\eta}\big(Qx^k+q\big)\right).
 \end{equation}

Thus, in combination with the decomposition $Q=Q_1-Q_2$ described by option~(a) in Subsection~2.2, Algorithm~1 leads to the following inertial algorithm for the IQP.

\medskip
\textbf{InDCA1}

\noindent \rule[0.05cm]{14.6cm}{0.01cm}

1:\ \, Choose $\eta>0$ satisfying $\eta>\lambda_n(Q)$. Put $\rho=\eta-\lambda_n(Q)$

2:\ \, Let $x^0\in C$, ${\rm Tol}\geq 0,$ and $\gamma\in \big[0,2^{-1}\rho\big)$ be given. Set $x^{-1}=x^0$

3:\ \, \textbf{for} $k=0,1,2,\dots $ \textbf{do}

4:\ \,\hskip0.5cm Set $d^k=\gamma(x^k-x^{k-1})$ and find $x^{k+1}\in C$ by formula~\eqref{iteration_InDCA1} 

5:\ \, \hskip0.5cm \textbf{if} $\|x^{k+1}-x^k\|\leq {\rm Tol}$ and $\|d^k\|\leq {\rm Tol}$ \textbf{then}

6:\ \, \hskip1cm Stop and return $(x^k,f(x^k))$

7:\ \, \hskip0.5cm  \textbf{end if}

8:\ \, \textbf{end for}\\
\noindent \rule[0.05cm]{14.6cm}{0.01cm}

\subsection{Inertial Proximal DCA}

Let $\eta>0$ be chosen according to rule~(b) in Subsection~2.2 and $\gamma$ be selected as in Algorithm~1 where $\lambda=0$. Consider the optimization problem
\begin{equation}\label{optim_(b)}
	\min\{\varphi(x)\mid x\in C\}
\end{equation} with 
\begin{equation*}\label{varphi} \varphi(x):=\left(\dfrac{1}{2}x^TQx+q^Tx\right)+\dfrac{\eta}{2}\|x\|^2-\eta \langle x^k,x\rangle -\gamma\langle x^k-x^{k-1},x\rangle\end{equation*} for $x\in\mathbb R^n$
and $C$ as in~\eqref{C}. Since $\eta>-\lambda_1(Q)$, we see that~\eqref{optim_(b)} is a strongly convex optimization problem with a closed and nonempty constraint set. Hence, it has a unique solution, which is defined by the condition $0\in \nabla\varphi(x)+N_C(x).$ Since the last inclusion can be rewritten equivalently as
\begin{equation*}\label{basic(b)_1} \gamma \big(x^k-x^{k-1}\big)\in \eta \big(x-x^k\big)+\big(Qx+q\big)+N_C(x), 
\end{equation*} we conclude that the point $x^{k+1}\in C$ satisfying~\eqref{basic(b)} coincides with the unique solution of~\eqref{optim_(b)}. Therefore, with the decomposition $Q=Q_1-Q_2$ described by option~(b) in Subsection~2.2, Algorithm~1 provides us with the next inertial algorithm for the IQP.

\medskip
\textbf{InDCA2}

\noindent \rule[0.05cm]{14.6cm}{0.01cm}

1:\ \, Choose $\eta>0$ satisfying $\eta>-\lambda_1(Q)$. Put $\rho=\eta+\lambda_1(Q)$

2:\ \, Let $x^0\in C$, ${\rm Tol}\geq 0,$ and $\gamma\in \big[0,2^{-1}\rho\big)$ be given. Set $x^{-1}=x^0$

3:\ \, \textbf{for} $k=0,1,2,\dots $ \textbf{do}

4:\ \,\hskip0.5cm Set $d^k=\gamma(x^k-x^{k-1})$ and solve the strongly convex quadratic program~\eqref{optim_(b)} to get the  unique solution denoted by $x^{k+1}\in C$ 

5:\ \, \hskip0.5cm \textbf{if} $\|x^{k+1}-x^k\|\leq {\rm Tol}$ and $\|d^k\|\leq {\rm Tol}$ \textbf{then}

6:\ \, \hskip1cm Stop and return $(x^k,f(x^k))$

7:\ \, \hskip0.5cm  \textbf{end if}

8:\ \, \textbf{end for}\\
\noindent \rule[0.05cm]{14.6cm}{0.01cm}

\section{Convergence Theorems}\label{Sect_4}

If $\inf\limits_{x\in C} f(x)$ is finite, then~\eqref{QP problem} has a global solution by the Frank-Wolfe theorem (see, e.g.,~\cite[Theorem~2.1]{LeeTamYen_book}); hence $C_*$ is nonempty.  Theorem~\ref{Thm_convergence_1} applies also to the case where $C$ is unbounded. However, in that case, we don't know whether the iteration sequence $\{x^k\}$ in question has a cluster point, or not. The following theorem provides a mild sufficient condition for the boundedness of all iteration sequences generated either by the algorithm~\textbf{InDCA1} or by the algorithm~\textbf{InDCA2}.

\begin{theorem}\label{thm_new1}
Suppose that the objective function of the indefinite quadratic program~\eqref{QP problem} is bounded below on $C$ and suppose that for any unbounded pseudo-face $\mathcal{F}$ of~$C$ one has
\begin{equation}\label{QC}
Q\Big((0^+F)\setminus\{0\}\Big)\cap (-N_{\mathcal F})=\emptyset,
\end{equation} where $N_{\mathcal F}$ is the normal cone of~$\mathcal{F}$ and~$0^+F$ is the recession cone of the face $F$ corresponding to~$\mathcal{F}$. Then, all iteration sequences generated either by~\textbf{InDCA1} or by~\textbf{InDCA2} are bounded. As a consequence, each iteration sequence has a cluster point, which is a KKT point of~\eqref{QP problem}. 
\end{theorem} 
\begin{proof} First, suppose that $\{x^k\}$ is an iteration  sequence generated by the algorithm~\textbf{InDCA1}. If $\{x^k\}$ is unbounded, then by considering a subsequence (if necessary), we may assume that $\lim\limits_{k\to\infty}\|x^k\|=+\infty$. Since the number of pseudo-faces of~$C$ is finite, there exists a subsequence $\{k'\}$ of $\{k\}$ and a pseudo-face~$\mathcal{F}$ such that $\{x^{k'+1}\}\subset\mathcal{F}$. As~\eqref{iteration_InDCA1} is equivalent to~\eqref{basic(a)}, one has for each $k'$ the inclusion \begin{equation}\label{basic(a)_1} \gamma \big(x^{k'}-x^{k'-1}\big)\in \eta \big(x^{k'+1}-x^{k'}\big)+\big(Qx^{k'}+q\big)+N_{\mathcal F},
	\end{equation} where $N_{\mathcal F}=N_C(x^{k'+1})$ is the normal cone of~$\mathcal{F}$. Without loss of generality we may assume that $\lim\limits_{k'\to\infty}\dfrac{x^{k'+1}}{\|x^{k'+1}\|}=v$ with $\|v\|=1$. Select a point $u\in F$, where $F$ is the face corresponding to~$\mathcal{F}$, such that $u\neq x^{k'+1}$ for all $k'$. Since
$$\lim\limits_{k'\to\infty}\dfrac{x^{k'+1}-u}{\|x^{k'+1}-u\|}=\lim\limits_{k'\to\infty}\left(\dfrac{x^{k'+1}}{\|x^{k'+1}\|}. \dfrac{\|x^{k'+1}\|}{\|x^{k'+1}-u\|}-\dfrac{u}{\|x^{k'+1}-u\|}\right)=\lim\limits_{k'\to\infty}\dfrac{x^{k'+1}}{\|x^{k'+1}\|},$$ using~\cite[Lemma~2.10]{Huong_Yao_Yen_JOGO2020}
we can assert that $v\in 0^+F$. Divide both 
sides of~\eqref{basic(a)_1} by $\|x^{k'+1}\|$ and get  \begin{equation*}\gamma \dfrac{x^{k'}-x^{k'-1}}{\|x^{k'+1}\|}\in \eta \dfrac{x^{k'+1}-x^{k'}}{\|x^{k'+1}\|}+\left(Q\Big(\dfrac{x^{k'}}{\|x^{k'+1}\|}\Big)+\dfrac{q}{\|x^{k'+1}\|}\right)+N_{\mathcal F}.
\end{equation*} Since $\lim\limits_{k\to\infty}\|x^{k+1}-x^k\|=0$ by assertion~(c) of Theorem~\ref{Thm_convergence_1}, passing the last inclusion to limit as $k'\to\infty$ yields $0\in Qv+N_{\mathcal F}$. We have arrived at a contradiction, because this contradicts the assumption~\eqref{QC}. The boundedness of $\{x^k\}$ has been proved.

Now, let $\{x^k\}$ be an iteration sequence generated by the algorithm~\textbf{InDCA2}. If $\{x^k\}$ is unbounded, then there exists a subsequence $\{k'\}$ of $\{k\}$ and a pseudo-face~$\mathcal{F}$ such that $\{x^{k'+1}\}\subset\mathcal{F}$ and $\lim\limits_{k\to\infty}\|x^{k'+1}\|=+\infty$. Getting from~\eqref{basic(b)} the inclusion \begin{equation*}\label{basic(b)_1} \gamma \big(x^{k'}-x^{k'-1}\big)\in \eta \big(x^{k'+1}-x^{k'}\big)+\big(Qx^{k'+1}+q\big)+N_{\mathcal F}
\end{equation*} with $N_{\mathcal F}=N_C(x^{k'+1})$ being the normal cone of~$\mathcal{F}$, and arguing similarly as above, we can reach a contradiction again. Thus, the sequence $\{x^k\}$ must be bounded.

The second assertion of the theorem is immediate from the first one and assertion~(d) of Theorem~\ref{Thm_convergence_1}.
\end{proof}

Our second theorem on the convergence of the algorithms~\textbf{InDCA1} and~\textbf{InDCA2} can be formulated as follows. 

\begin{theorem}\label{thm_new2}
	 Under the assumptions of Theorem~\ref{thm_new1}, for each iteration sequence $\{x^k\}$ generated either by~\textbf{InDCA1} or by~\textbf{InDCA2}, there exists a connected component $C_{i_0}$ of $C_*$ such that
	 \begin{equation}\label{distance}
	 \lim\limits_{k\to\infty} d\big(x^k,C_{i_0}\big)=0.
	 \end{equation} Therefore, the objective function of~\eqref{QP problem} has the same value at all the cluster points of $\{x^k\}$. 
\end{theorem} 
\begin{proof} Suppose that $\{x^k\}$ is an iteration sequence generated either by~\textbf{InDCA1} or by~\textbf{InDCA2}. Let $C_1, C_2, \cdots, C_r$ be the connected components of $C_*$ (see Lemma~\ref{Lemma 3}). Choose $\delta>0$ such that~\eqref{positive_excess} holds. For each $i\in\{1,\ldots,r\}$, since the distance function $d(.,C_i)$ satisfies the global Lipschitz condition $$\big|d\big(x,C_i\big)-d\big(u,C_i\big)\big|\leq \|x-u\|\quad \forall x,u\in\mathbb R^n$$ (see~\cite[Proposition~2.4.1]{Clarke_1990}), it is continous on $\mathbb R^n$. Hence, the sets \begin{equation}\label{Omega_i}
	\Omega_i:=\left\{x\in\mathbb R^n\mid d\big(x,C_i\big) <\dfrac{\delta}{4}\right\}\quad (i=1,\ldots,r)
	\end{equation} are open. 
	
	Next, using assertion~(c) of Theorem~\ref{Thm_convergence_1}, one can find an index $k_1\in\mathbb N$ such that \begin{equation}\label{k1} \|x^{k+1}-x^k\|<\dfrac{\delta}{4}\quad\;  \forall k\geq k_1.
	\end{equation}
	
	Furthermore, the sequence $\{x^k\}$ is bounded by Theorem~\ref{thm_new1}. If there are infinitely many members of the sequence lying in the set $D:=\mathbb R^n\setminus \left(\bigcup_{i=1}^r\Omega_i\right)$, then there exists a subsequence of~$\{x^k\}$ converging to a point $\bar x$. Thanks to assertion~(d) of Theorem~\ref{Thm_convergence_1} and~\eqref{basic_inclusions_IQP}, we can assert that $\bar x\in C_*$. But this is impossible because $D$ is closed and $C_*\subset \bigcup_{i=1}^r\Omega_i$. We have thus proved that there are only finitely many members of the sequence $\{x^k\}$ lying outside the union $\bigcup_{i=1}^r\Omega_i$. So, by considering a larger index $k_1$ (if necessary), we can assume that \begin{equation}\label{property1} x^k\in \bigcup_{i=1}^r\Omega_i\quad\; \forall k\geq k_1.
	\end{equation} Clearly, this assures that $x^{k_1}\in\Omega_{i_0}$ for some $i_0\in\{1,\ldots,r\}$.

\smallskip
{\sc Claim 1.} \textit{One has $\{x^k\}_{k\geq k_1}\subset \Omega_{i_0}$.}

\smallskip
Indeed, if $x^{k_1+1}\notin \Omega_{i_0}$, then by~\eqref{property1} we can find $i_1\in\{1,\ldots,r\}$ with $i_1\neq i_0,$ such that $x^{k_1+1}\in\Omega_{i_1}$. From~\eqref{Omega_i} it follows that $\|x^{k_1+1}-u\|<\dfrac{\delta}{4}$ for some $u\in C_{i_1}$. Then, by~\eqref{Omega_i} and~\eqref{k1} we have 
$$d\big(u,C_{i_0}\big)\leq \|u-x^{k_1+1}\|+\|x^{k_1+1}-x^{k_1}\|+d\big(x^{k_1},C_{i_0}\big)<\dfrac{3\delta}{4}.$$ This contradicts~\eqref{positive_excess}. We have thus shown that $x^{k_1+1}\in\Omega_{i_0}$. Now, letting $x^{k_1+1}$ play the role of the above element $x^{k_1}$, we can repeat the reasoning to prove that $x^{k_1+2}\in\Omega_{i_0}$, and so on. Therefore, the property $\{x^k\}_{k\geq k_1}\subset \Omega_{i_0}$ holds true.
	
 The sets $\Omega_1,\ldots,\Omega_r$ are pairwise distinct. Indeed, suppose for a while that there exists a point $y\in \Omega_{i_1}\cap \Omega_{i_2}$, where $i_1\neq i_2$. Then, by~\eqref{Omega_i} we can find $x\in C_{i_1}$ and $u\in C_{i_2}$ such that $\|y-x\|<\dfrac{\delta}{4}$ and $\|y-u\|<\dfrac{\delta}{4}$. Then, we get $$\dfrac{\delta}{2}>\|y-x\|+\|y-u\|\geq\|x-u\|\geq d(x,C_{i_2}).$$ This certainly contradicts~\eqref{positive_excess}.  
 
 \smallskip
 {\sc Claim 2.} \textit{For the above-chosen index $i_0\in \{1,\ldots,r\}$, the equality~\eqref{distance} holds.}
 
 \smallskip 
 Indeed, if~\eqref{distance} was false, then we would find a number $\varepsilon>0$ and a convergent subsequence $\{x^{k'}\}$ of $\{x^k\}$ such that \begin{equation}\label{property2} d\big(x^{k'},C_{i_0}\big)\geq\varepsilon\quad\; \forall k'.	\end{equation} 
 Denote the limit point of $\{x^{k'}\}$ by $\bar x$ and get $\bar x\in C_*$ by assertion~(d) of Theorem~\ref{Thm_convergence_1} and~\eqref{basic_inclusions_IQP}. Passing the inequality in~\eqref{property2} to limit as $k'\to\infty$ gives $d\big(\bar x,C_{i_0}\big)\geq\varepsilon$. In particular, $\bar x\notin C_{i_0}$. On one hand, this implies that $\bar x\in \bigcup_{i\neq i_0} C_i$. On the other hand, combining Claim~1 with the fact that the sets $\Omega_1,\ldots,\Omega_r$ are pairwise distinct, we get $\{x^k\}_{k\geq k_1}\subset D_{i_0}$, where the closed set  $D_{i_0}$ is defined by $D_{i_0}:=\mathbb R^n\setminus \left(\bigcup_{i\neq i_0}\Omega_i\right)$. Since $\bar x=\lim\limits_{k'\to\infty}x^{k'}\in D_{i_0}$ and $D_{i_0}\cap  \left(\bigcup_{i\neq i_0} C_i\right)=\emptyset$, we have arrived at a contradiction. Thus,~\eqref{distance} is valid.
 
 From~\eqref{distance} and the continuity of the function $d\big(.,C_{i_0}\big)$ one can easily deduce that any cluster point of $\{x^k\}$ must belong to $C_{i_0}$. Now, the second assertion of the theorem is an immediate consequence of the property~(c) in Lemma~\ref{Lemma 3}. 
 
 The proof is complete.	
\end{proof}

The next useful example helps us to better understand the qualification condition~\eqref{QC} and the assertions of the above two theorems.  

\begin{example} \rm
	Consider problem~\eqref{QP problem} with $n=2, m = 3$,\\  $$Q=\begin{bmatrix} 
		2&0 \\ 0&-2
	\end{bmatrix},\ \; A=\begin{bmatrix}
		1&-1 \\ 
		1&1\\
		1&0
	\end{bmatrix},\ \; q=\begin{pmatrix}
		0\\
		0 
	\end{pmatrix},\ \; b=\begin{pmatrix}
		0\\	
		0\\
		\dfrac{1}{4}
	\end{pmatrix}.$$	
	With these data, the IQP becomes
	\begin{eqnarray}\label{QP_example}
		\min\left\{x_1^2-x_2^2 \mid x=(x_1,x_2)\in\mathbb{R}^2,\; x\in C\right\}, \end{eqnarray}
	where $C=\left\{x=(x_1,x_2)\in \mathbb{R}^2 \mid x_1 \geq |x_2|, \; x_1\geq \dfrac{1}{4}\right\}.$
	\begin{figure}[!ht]
		\centering
		\includegraphics[scale=0.5]{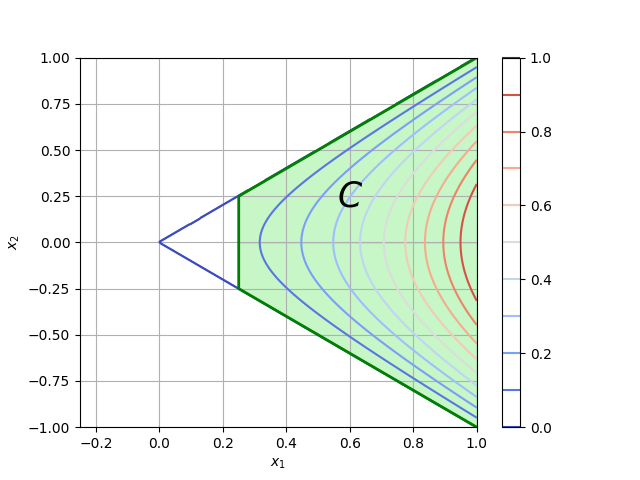} 
		\caption{The constraint set $C$ and the contour of the objective function}\label{fig1}
	\end{figure}
	
	The pseudo-faces of $C$ are
	\begin{equation*}
		\begin{array}{ll}
			\mathcal{F}_1&= \left\{x=(x_1,x_2)\in\mathbb{R}^2 \mid x_1-x_2= 0,\; x_1 > \dfrac{1}{4}\right\},\\
			\mathcal{F}_2&= \left\{x=(x_1,x_2)\in\mathbb{R}^2 \mid x_1+x_2= 0,\; x_1 > \dfrac{1}{4}\right\},\\
			\mathcal{F}_3&= \left\{x=(x_1,x_2)\in\mathbb{R}^2 \mid x_1 = \dfrac{1}{4},\; x_1> |x_2|,\right\},\\
			\mathcal{F}_4&= \left\{x=(x_1,x_2)\in\mathbb{R}^2 \mid x_1-x_2 > 0,\; x_1+x_2>0,\; x_1 > \dfrac{1}{4}\right\}.
		\end{array}
	\end{equation*}
     The faces corresponding to the above pseudo-faces respectively are
	\begin{equation*}
		\begin{array}{ll}
			F_1&=\left\{x=(x_1,x_2)\in\mathbb{R}^2 \mid x_1-x_2= 0,\; x_1\geq \dfrac{1}{4}\right\},\\
			F_2&=\left\{x=(x_1,x_2)\in\mathbb{R}^2 \mid x_1+x_2= 0,\; x_1\geq \dfrac{1}{4}\right\},\\
			F_3&= \left\{x=(x_1,x_2)\in\mathbb{R}^2 \mid x_1 = \dfrac{1}{4},\; x_1\geq |x_2|\right\},\\
			F_4&= C.
		\end{array}
	\end{equation*}
	It is easy to verify that ${\mathcal S}_1={\mathcal S}=F_1\cup F_2$, meanwhile $${\mathcal S}_3={\mathcal S}_2=C_*=F_1\cup F_2\cup \Big\{\big(\dfrac{1}{4},0\big)\Big\}.$$  
	
	Now, let us analyze algorithm~\textbf{InDCA1}. Since $\lambda_1(Q)=-2$ and $\lambda_2(Q)=2$, one can choose $\eta =3$ and $\gamma =\dfrac{1}{3}$. Then, formula~\eqref{iteration_InDCA1} becomes
	\begin{equation*}
		x^{k+1}=P_C\left(\Big(1+\frac{1}{6}\Big)x^k-\frac{1}{6}x^{k-1}-\frac{1}{3}\Big(\begin{bmatrix}
			2 & 0 \\ 
			0 & -2
		\end{bmatrix}x^k\Big)\right)
	\end{equation*}
	or, equivalently,
	\begin{equation} \label{iteration_InDCA1_example}
		\begin{pmatrix}
			x^{k+1}_1 \\ 
			x^{k+1}_2
		\end{pmatrix} = P_C \left(\begin{pmatrix}
			\dfrac{1}{9}\left(4x^k_1-x^{k-1}_1\right) \\[0.75em]
			\dfrac{1}{9}\left(16x^k_2-x^{k-1}_2\right)
		\end{pmatrix}\right).
	\end{equation}
	Consider the following cases.
	
	\noindent\textbf{Case 1:} Take any $x^0\in F_1$. So, $ x^{-1}=x^{0} = \left(\dfrac{1}{4} +t, \dfrac{1}{4}+t\right)$ for some fixed $t\geq0$. 	By~\eqref{iteration_InDCA1_example}, we have
	\begin{equation*}
		x^1= P_C \left(\begin{pmatrix}
			\dfrac{1}{12}+\dfrac{1}{3}t \\[0.75em]
			\dfrac{5}{12}+\dfrac{5}{3}t
		\end{pmatrix}\right)= \begin{pmatrix}
			\dfrac{1}{4} +t\\[0.75em]
			\dfrac{1}{4} +t
		\end{pmatrix}.
	\end{equation*}
	Thus, if the initial point belongs to $F_1$, then~\textbf{InDCA1} stops after 1 step and $x^1$ is a global solution of~\eqref{QP_example}. By symmetry, for any $x^0\in F_2$, the algorithm stops after 1 step yielding a global solution.
	
	\noindent\textbf{Case 2:} Let $x^0= \left(\dfrac{1}{4},0\right) \in \mathcal{F}_3$. It follows from~\eqref{iteration_InDCA1_example} that 
	$$x^1=P_C \left(\begin{pmatrix}
		\dfrac{1}{12}\\[0.75em]
		0
	\end{pmatrix}\right)= \begin{pmatrix}
		\dfrac{1}{4} \\[0.75em]
		0
	\end{pmatrix}.$$
	Then, algorithm~\textbf{InDCA1} stops after 1 step with $x^1$ being a KKT point, which is not a local solution.
	
	\noindent\textbf{Case 3:} Choose $x^{0} = \left(\dfrac{1}{4}, \dfrac{1}{8}\right) \in \mathcal{F}_3$. 	It follows from~\eqref{iteration_InDCA1_example} that $x^{1}= P_C \left((\frac{1}{12}, \frac{5}{24})\right).$			Since $\frac{1}{12}<\frac{1}{4}$ and $\frac{5}{24}< \frac{1}{4}$, the metric projection of $(\frac{1}{12}, \frac{5}{24})$ on $C$ must lie on the face $F_3$. Hence, 
	$x^1=\left(\frac{1}{4}, \frac{5}{24}\right).$
	By~\eqref{iteration_InDCA1_example} we then get $x^2= P_C\left(\frac{1}{12},\frac{77}{216}\right).$ Thus, $x^2\in F_1$ as $\frac{1}{12} >0$ and $\frac{77}{216}>\frac{1}{4}$. To prove that  $x^2=\left(\frac{1}{4}, \frac{1}{4}\right)$, we use the characterization of the metric projection in~\cite[Chapter~I, Theorem~2.3]{KS1980}. Since the scalar product of the vectors $\left(\frac{1}{12}-\frac{1}{4}, \frac{77}{216}-\frac{1}{4}\right)$ and $(1,1)$ is negative, the metric projection of $\left(\frac{1}{12},\frac{77}{216}\right)$ on $C$ is $x^2=\left(\frac{1}{4},\frac{1}{4}\right)$. Similarly, we can find
	$$x^3=\left(\frac{109}{432},\frac{109}{432}\right).$$
	Next, let us show that $x^{k+1}_1 = x^{k+1}_2 > x^k_1 = x^k_2 \geq\frac{1}{4} $ for all $k\geq 2$ by induction. Clearly, $x^3_1 = x^3_2 > x^2_1=x^2_2 =\frac{1}{4}$. Suppose that $x^{k+1}_1 = x^{k+1}_2 > x^k_1 = x^k_2 \geq\frac{1}{4} $ for some $k\geq 3$. Then, $x^k=(t_k,t_k)$ and $x^{k-1}=(t_{k-1}, t_{k-1})$ for some $t_k$ and $t_{k-1}$ with $t_k > t_{k-1}\geq \frac{1}{4}$. By~\eqref{iteration_InDCA1_example}, one has
	$$\begin{pmatrix}
		x^{k+1}_1 \\ 
		x^{k+1}_2
	\end{pmatrix} = P_C \left(\begin{pmatrix}
		\dfrac{1}{9}\left(4t_k-t_{k-1}\right) \\[0.75em]
		\dfrac{1}{9}\left(16t_k-t_{k-1}\right)
	\end{pmatrix}\right).$$
	It follows from our induction assumption that $\frac{1}{9}(4t_k-t_{k-1}) >0$, $\frac{1}{9}(16t_k-t_{k-1}) > \frac{1}{4}$, and $\frac{1}{9}(16t_k-t_{k-1}) > \frac{1}{9}(4t_k-t_{k-1})$. Hence, $x^{k+1}\in F_1$, and therefore, $x^{k+1}$ has the form $(t_{k+1}, t_{k+1})$ with $t_{k+1}\geq \frac{1}{4}$. By~\cite[Chapter~I, Theorem~2.3]{KS1980}, we have
	\begin{equation*}
		\left\langle \left(\dfrac{1}{9}(4t_k-t_{k-1}) - t_{k+1}, \dfrac{1}{9}(16t_k-t_{k-1}) -t_{k+1}\right), (1,1)\right\rangle \leq 0.
	\end{equation*}
	Thus, $t_{k+1}\geq \dfrac{10}{9}t_k -\dfrac{1}{9} t_{k-1}.$	Since $t_k > t_{k-1}$, the latter implies that $t_{k+1} > t_k$. So, $x^{k+1}_1 = x^{k+1}_2 > x^k_1 = x^k_2 $ for all $k\geq 2$ and the sequence $\{t_k\}_{k=2}^\infty$ is strictly increasing. Therefore, $x^{k+1} \in \mbox{\rm int} F_1$ for all $k\geq 3$. Hence, by the characterization of the metric projection, we obtain
	\begin{equation*}
		\left\langle \left(\dfrac{1}{9}(4t_k-t_{k-1}) - t_{k+1}, \dfrac{1}{9}(16t_k-t_{k-1}) -t_{k+1}\right), (1,1)\right\rangle = 0,\quad\forall k\geq 3.
	\end{equation*}
	So,
	\begin{equation*} 
		t_{k+1}= \dfrac{10}{9}t_k -\dfrac{1}{9} t_{k-1} = t_k+\dfrac{1}{9} (t_k-t_{k-1}),\quad\forall k\geq 3.
	\end{equation*}
	Thus, for all $k\geq 3$  one has 
	\begin{equation*}
		t_{k+1} -t_k= \dfrac{1}{9} (t_k-t_{k-1})= \dfrac{1}{9^{k-2}}(t_3-t_2).
	\end{equation*}
	It follows that
	\begin{equation*}
		\begin{array}{ll}
			t_{k+1} = t_k+\dfrac{1}{9^{k-2}} (t_3-t_2)&=t_{k-1}+(t_3-t_2)\left(\dfrac{1}{9^{k-3}}+\dfrac{1}{9^{k-2}}\right)\\ &= t_2 +(t_3-t_2)\sum\limits_{i=0}^{k-2}\dfrac{1}{9^i}\\[0.75em]
			&< t_2 +(t_3-t_2)\sum\limits_{i=0}^{\infty}\dfrac{1}{9^i}\\
			&= t_2+\dfrac{9}{8}(t_3-t_2)\\
			&=\dfrac{9}{8} t_3 - \dfrac{1}{8}t_2= \dfrac{97}{384}.
		\end{array}
	\end{equation*}
	Therefore, $\{t_k\}_{k=2}^\infty$ is bounded. Combining this with the monotonicity of $\{t_k\}_{k=2}^\infty$, we can deduce that the sequence $\{x^{k+1}\}$ converges to some point $\bar{x}$. Since $F_1$ is closed, $\bar{x}\in F_1$. Consequently, algorithm~\textbf{InDCA1} stops after a finite number of iterations at a global solution if $\mbox{\rm Tol} >0$.
	
	\begin{figure}[!ht]
		\centering
		\includegraphics[scale=0.5]{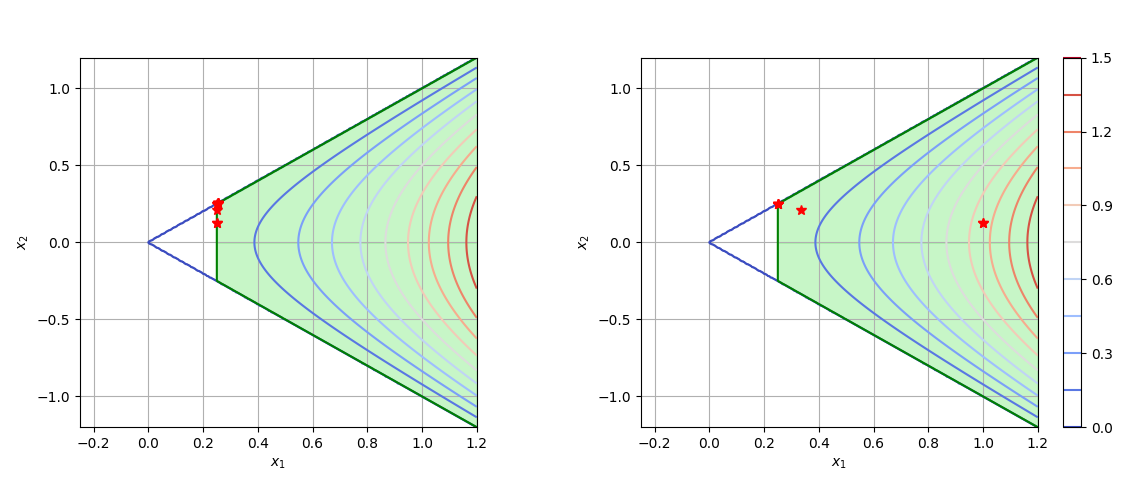} 
		\caption{Numerical simulations for Case 3 and Case 5}\label{fig2}
	\end{figure}
	
	\noindent\textbf{Case 4:} Let $x^0=(1,0)\in \mathcal{F}_4$. By~\eqref{iteration_InDCA1_example} we get $x^1=P_C \left(\frac{1}{3},0\right)=\left(\frac{1}{3},0\right)$.	It follows that $x^2= P_C\left(\frac{1}{27},0\right)=\left(\frac{1}{4},0\right)$, and hence $x^3=P_C\left(\frac{1}{27},0\right)=\left(\frac{1}{4},0\right)$. Thus, algorithm~\textbf{InDCA1} stops after 3 steps at a KKT point, which is not a local solution.
	
	\noindent\textbf{Case 5:} Choose $x^0=\left(1,\frac{1}{8}\right)\in \mathcal{F}_4$. Then, $x^1=\left(\frac{1}{3},\frac{5}{24}\right)$ and $x^2=x^3=x^4=\left(\frac{1}{4},\frac{1}{4}\right)$. So, algorithm~\textbf{InDCA1} stops after 4 steps at a global solution.
	
So, in the above 5 cases, all the iteration sequences in question are bounded. However, our IQP does not satisfies the qualification condition~\eqref{QC}. Indeed, consider the unbounded pseudo-face $\mathcal{F}_1$ and its corresponding face $F_1$. We have
	\begin{align*}
		Q\Big((0^+F_1)\setminus\{0\}\Big) & = \{(t,-t)\in\mathbb{R}^2\mid t >0\},
	\end{align*}
	and
	\begin{align*}
		N_{{\mathcal F}_1} & = \{(-t,t)\in\mathbb{R}^2\mid t \geq0\}.
	\end{align*}
	Then,
	\begin{align*}
		Q\Big((0^+F_1)\setminus\{0\}\Big) &\cap (-N_{{\mathcal F}_1} ) =\{(t,-t)\in\mathbb{R}^2\mid t >0\}.
	\end{align*}
	Therefore, the condition~\eqref{QC} does not hold.
\end{example}

The above example gives rise to a natural question: \textit{Whether the conclusions of Theorems~\ref{thm_new1} and~\ref{thm_new2} still hold without the qualification condition~\eqref{QC}?}

\section*{Acknowledgements}

The authors were supported by Hanoi University of Industry, the National Research Foundation of Korea (NRF) grant funded by the Korea government (MEST), and Vietnam Academy of Science and Technology. N.N. Thieu and N.D. Yen would like to thank the Sungkyunkwan University for supporting their one-month research stay in Suwon.


\begin{thebibliography}{99}
	

\bibitem{Akoa_2008} Akoa, F.B.: \textit{Combining DC Algorithms (DCAs) and decomposition techniques for the training of nonpositive–semidefinite kernels}. IEEE Trans. Neur. Networks. 2008; 19:1854--1872.  



\bibitem{Bomze_1998}  Bomze, I.M.: \textit{On standard quadratic optimization problems}. J. Global Optim. 1998; 13: 369--387.

\bibitem{Bonnans_Shapiro_2000} Bonnans, J.F., Shapiro, A.: Perturbation Analysis of Optimization Problems, Springer-Verlag, New York, 2000. 

\bibitem{Bomze_Danninger_1994} Bomze, I.M., Danninger, G.: \textit{A finite algorithm for solving general quadratic problems}. J. Global Optim. 1994; 4:1--16.

\bibitem{Cambini_Sodini_2005}  Cambini, R., Sodini, C.: \textit{Decomposition methods for solving nonconvex quadratic programs via Branch and Bound}. J. Global Optim. 2005; 33:313--336. 

\bibitem{Clarke_1990} Clarke, F.H.: Optimization and Nonsmooth Analysis, SIAM, Philadelphia, 1990.	 

\bibitem{Cornuejols_2018}  Cornu\'ejols, G., Pe\~na, J., T\"ut\"unc\"u, R.: Optimization Methods in Finance, Second edition, Cambridge University Press, Cambridge, 2018.

\bibitem{CLY_2024} Cuong,  T.H., Lim, Y., Yen, N.D.: \textit{On a solution method in indefinite quadratic programming under linear constraints}, Optimization 2024; 73:1087--1112.

\bibitem{de Oliveira_Tcheou_2019} de Oliveira, W.,  Tcheou, M.P.: \emph{An inertial algorithm for DC programming}, Set-Valued and Variational Analysis 2019; 27: 895--919.

\bibitem{de Oliveira_2020} de Oliveira, W.: \textit{The ABC of DC programming}, Set-Valued Var. Anal. 28 (2020), 679--706.

\bibitem{Gould_Toint}  Gould, N.I.M., Toint, Ph.L.: \textit{A Quadratic Programming Page}. http://www.numerical.rl.ac.uk/people/nimg/qp/qp.html

\bibitem{Gupta_1995} Gupta, O.K.: \textit{Applications of quadratic programming}. J. Inf. Optim. Sci. 1995; 16:177--194. 

\bibitem{Huong_Yao_Yen_JOGO2020} Huong, N.T.T., Yao, J.-C., Yen, N.D.: \textit{Geoffrion's proper efficiency in linear fractional vector optimization with unbounded constraint sets}. J. Global Optim. 2020; 78: 545--562.

\bibitem{Ioffe_Tihomirov_1979} Ioffe,  A.D.,  Tihomirov, V.M.: Theory of Extremal Problems, North-Holland Publishing Co., Amsterdam-New York, 1979.



\bibitem{KS1980} Kinderlehrer, D., Stampacchia, G.: An Introduction to Variational Inequalities and Their Applications, Academic Press, Inc., New York-London, 1980.

\bibitem{LeeTamYen_book}  Lee, G.M., Tam, N.N., Yen, N.D.: Quadratic Programming and Affine Variational Inequalities: A Qualitative Study, Springer Verlag, New York, 2005.


\bibitem{LeThi_PhamDinh_AOR05} Le Thi, H.A., Pham Dinh, T.: \textit{The DC (difference of convex functions) programming and DCA revisited with DC models of real world nonconvex optimization problems}. Ann. Oper. Res. 2005; 133:23--46.


\bibitem{ATY2}  Le Thi, H.A., Pham Dinh, T., Yen, N.D.: \textit{Properties of two DC algorithms in quadratic programming}. J. Global Optim. 2011; 49:481--495.  




\bibitem{Liu et al_2017a} Liu, F., Huang, X., Peng, C., Yang, J., Kasabov, N.:  \textit{Robust kernel approximation for classification}. The 24th International Conference ``Neural Information Processing", ICONIP 2017, Guangzhou, China, November 14--18, 2017, Proceedings, Part I; 2017.p. 289--296.


\bibitem{Luo_Tseng_1992}  Luo, Z.-Q., Tseng, P.: \textit{Error bound and convergence analysis of matrix splitting algorithms for the affine variational inequality problem}. SIAM J. Optim. 1992; 2:43--54. 

\bibitem{McCarl et al_1977}  McCarl, B.A., Moskowitz, H., Furtan, H.: \textit{Quadratic programming applications}. Omega. 1977; 5:43--55.

\bibitem{Pardalos_Vavasis_1991} Pardalos, P.M., Vavasis, S.A.: \textit{Quadratic programming with one negative eigenvalue is NP-hard}. J. Global Optim. 1991; 1:15--22. 

\bibitem{PhamDinh_LeThi_AMV97}  Pham Dinh, T., Le Thi, H.A.: \textit{Convex analysis approach to d.c. programming: theory, algorithms and applications}. Acta Math. Vietnam. 1997; 22: 289--355.  

\bibitem{PhamDinh_LeThi_2} Pham Dinh, T., Le Thi,  H.A.: \textit{Solving a class of linearly constrained indefinite quadratic programming problems by d.c. algorithms}. J. Global Optim. 1997; 11:253--285. 

\bibitem{PhamDinh_LeThi98}  Pham Dinh, T., Le Thi, H.A.: \textit{A d.c. optimization algorithm for solving the trust-region subproblem}. SIAM J. Optim. 1998; 8:476--505.  

\bibitem{PhamDinh_LeThi_3}  Pham Dinh, T., Le Thi, H.A.: \textit{A branch and bound method via DC optimization algorithm and ellipsoidal techniques for box constrained nonconvex quadratic programming problems}. J. Global Optim. 1998; 13:171--206.

\bibitem{PhamDinh_LeThi_4}  Pham Dinh, T., Le Thi, H.A.: \textit{DC (difference of convex functions) programming. Theory, algorithms, applications: The state of the art}. Proceedings of the First International Workshop on Global Constrained Optimization and Constraint Satisfaction (Cocos'02), Valbonne Sophia Antipolis, France, October. 2002; 2--4.

\bibitem{PhamDinh_LeThi_Akoa} Pham Dinh, T., Le Thi, H.A., Akoa, F.: \textit{Combining DCA (DC Algorithms) and interior point techniques for large-scale nonconvex quadratic programming}. Optim. Methods Softw. 2008; 23:60--629.

\bibitem{PBGM_1962} Pontryagin, L.S., Boltyanskii, V.G., Gamkrelidze, R.V., ~Mishchenko, E.F.: The Mathematical Theory of Optimal Processes, Interscience Publishers, New York--London, 1962. 


\bibitem{Polak_1997} Polak, E.: Optimization. Algorithms and Consistent Approximations, Springer-Verlag, New York, 1997.

\bibitem{Roc70} Rockafellar, R.T.: Convex Analysis, Princeton University Press, NJ, 1970. 

\bibitem{Stoer_Bulirsch_1980}  Stoer, J., Bulirsch, R.: \textit{Introduction to Numerical Analysis}. Springer-Verlag, New York; 1980. 


\bibitem{Tuan_JOTA2015} Tuan, H.N.: \textit{Boundedness of a type of iterative sequences in two-dimensional quadratic programming}. J. Optim. Theory Appl. 2015; 164:234--245.

\bibitem{Tuan_JMAA2015} Tuan, H.N.: \textit{Linear convergence of a type of DCA sequences in nonconvex quadratic programming}. J. Math. Anal. Appl. 2015; 423:1311--1319. 


\bibitem{Xu et al_2017} Xu, H.-M., Xue, H., Chen, X.-H., Wang, Y.-Y.: \textit{Solving indefinite kernel support vector machine with difference of convex functions programming}. Proceedings of the Thirty-First AAAI Conference on Artificial Intelligence (AAAI-17), Association for the Advancement of Artificial Intelligence; 2017.p. 2782--2788.

\bibitem{Xue et al_2019} Xue, H., Song, Y., Xu, H.-M.: \textit{Multiple indefinite kernel learning for feature selection}. Knowledge-Based Systems. 2020; 191:105272.

\bibitem{Ye89} Ye, Y.: \textit{An extension of Karmarkar's algorithm and the trust region method for quadratic programming}. In: N.~Megiddo (ed.), Progress in Mathematical Programming, pp.~49--63, Springer, New York, 1989.

\bibitem{Ye92}  Ye, Y.: \textit{On affine scaling algorithms for nonconvex quadratic programming}. Math.~Programming. 1992; 56:285--300.

\bibitem{Ye97} Ye, Y.: Interior Point Algorithms: Theory and Analysis, Wiley, New York, 1997.       
\end{thebibliography}
\end{document}